\documentclass[11pt]{article}
\usepackage{amsmath}
\usepackage{amsthm}
\usepackage{amssymb,amsfonts}
\usepackage{hyperref}
\usepackage{latexsym}
\usepackage{todonotes}
\usepackage{nicefrac}
\usepackage{epsfig}
\usepackage{stmaryrd}
\usepackage{setspace}
\usepackage{enumerate}
\usepackage[all]{xypic}
\usepackage{bbm,ifpdf,tikz}
\ifpdf
\usepackage{pdfsync}
\fi

\newtheorem{theorem}{Theorem}[section]

\newtheorem{lemma}[theorem]{Lemma}
\newtheorem{proposition}[theorem]{Proposition}
\newtheorem{corollary}[theorem]{Corollary}

\newtheorem{conjecture}[theorem]{Conjecture}

{
\theoremstyle{definition}

\newtheorem{example}[theorem]{Example}

\newtheorem{remark}[theorem]{Remark}
}

\newcommand{\excise}[1]{}

\renewcommand{\and}{\qquad\text{and}\qquad}
\newcommand{\Ind}{\operatorname{Ind}}

\newcommand{\Hom}{\operatorname{Hom}}

\newcommand{\Q}{\mathbb{Q}}
\newcommand{\N}{\mathbb{N}}
\newcommand{\R}{\mathbb{R}}
\newcommand{\C}{\mathbb{C}}

\newcommand{\cA}{\mathcal{A}}
\newcommand{\la}{\lambda}

\newcommand{\cs}{\C^\times}

\newcommand{\Conf}{\operatorname{Conf}}
\newcommand{\FI}{\operatorname{FI}}
\newcommand{\FB}{\operatorname{FB}}
\renewcommand{\Vec}{\operatorname{Vec}}

\begin{document}
\spacing{1.2}
\noindent{\Large\bf Equivariant log concavity and representation stability}\\

\noindent{\bf Jacob P. Matherne}\\
Mathematical Institute, University of Bonn, 53115 Bonn, Germany\\
and Max Planck Institute for Mathematics, 53111 Bonn, Germany
\vspace{.1in}

\noindent{\bf Dane Miyata}\footnote{Supported by NSF grants DMS-1954050 and DMS-2039316.}\\
Department of Mathematics, University of Oregon, Eugene, OR 97403
\vspace{.1in}

\noindent{\bf Nicholas Proudfoot}\footnote{Supported by NSF grants DMS-1954050, DMS-2039316, and DMS-2053243.}\\
Department of Mathematics, University of Oregon, Eugene, OR 97403
\vspace{.1in}

\noindent{\bf Eric Ramos}\footnote{Supported by NSF grants DMS-1704811 and DMS-2137628.}\\
Department of Mathematics, Bowdoin College, Brunswick, ME 04011\\

{\small
\begin{quote}
\noindent {\em Abstract.}
We expand upon the notion of equivariant log concavity, and make equivariant log concavity conjectures
for Orlik--Solomon algebras of matroids, Cordovil algebras of oriented matroids, and Orlik--Terao algebras of hyperplane arrangements.
In the case of the Coxeter arrangement for the Lie algebra $\mathfrak{sl}_n$,
we exploit the theory of representation stability to give computer assisted proofs of these conjectures in low degree.
\end{quote} }

\section{Introduction}
For any positive integer $n$ and any topological space $X$, let $\Conf(n,X)$ be the space of ordered configurations
of $n$ distinct points in $X$.  This space is equipped with an action of the symmetric group $\mathfrak{S}_n$, which acts by permuting the points.
If $G$ is a group acting on $X$, then the action of $\mathfrak{S}_n$ descends to an action on $\Conf(n,X)/G$.

Our main objects of study will be the following finite dimensional graded representations of $\mathfrak{S}_n$:
\begin{itemize}
\item $A_n := H^*(\Conf(n,\C); \Q)$.  This is also known as the Orlik--Solomon algebra of the braid matroid.
\item $B_n := H^*(\Conf(n,\C)/\cs; \Q)$, where $\cs$ acts on $\C$ by multiplication.  
This is also known as the reduced Orlik--Solomon algebra of the braid matroid. 
\item $C_n := H^{2*}(\Conf(n,\R^3); \Q)$.\footnote{This cohomology ring vanishes in odd degree, so we set $C_n^i$ equal to the cohomology
in degree $2i$.  We do the same in the definitions of $D_n$ and $M_n$ below.}  
This is also known as the Cordovil algebra of the oriented braid matroid.
\item $D_n := H^{2*}(\Conf(n,SU_2)/SU_2; \Q)$, where $SU_2$ acts on itself by left translation.
\end{itemize}

\begin{remark}\label{identify}
Identifying $\R^3$ with the complement of the identity in $SU_2$ induces a homeomorphism
$\Conf(n-1,\R^3)\cong \Conf(n,SU_2)/SU_2$, which is equivariant with respect to the action of the group $\mathfrak{S}_{n-1}\subset \mathfrak{S}_n$.  
It follows that the restriction of $D_n$ to $\mathfrak{S}_{n-1}$ is isomorphic to $C_{n-1}$.
\end{remark}

\begin{remark}
For any $d\geq 2$, the cohomology of $\Conf(n,\R^d)$ vanishes in all degrees that are not multiples of $d-1$, and we have
$\mathfrak{S}_n$-equivariant algebra isomorphisms
$$H^{(d-1)*}(\Conf(n,\R^d); \Q)\cong \begin{cases}
A_n &\text{if $d$ is even}\\ C_n &\text{if $d$ is odd.}\end{cases}$$
Thus we would not gain anything new by considering configuration spaces in Euclidean spaces of higher dimension.
This is due originally to Cohen \cite{Cohen}; see \cite[Corollary 5.6]{dLS} for a more modern treatment.
\end{remark}

There is one more graded representation that we will consider, whose definition is more technical.
Let $X_n$ denote the complex affine hypertoric variety associated with the root system of the Lie algebra $\mathfrak{sl}_n$;
see \cite[Section 3.1]{MPY} for an explicit description.  The variety $X_n$ comes equipped with an action of the symmetric group $\mathfrak{S}_n$,
and we consider the induced action on intersection cohomology:
\begin{itemize}
\item $M_n := I\!H^{2*}(X_n; \Q)$.  This can also be described as the quotient of the Orlik--Terao algebra of the $\mathfrak{sl}_n$
Coxeter arrangement by its canonical linear system of parameters.
\end{itemize}

\begin{remark}
For a more concrete perspective, we give explicit presentations of the rings $A_n$, $B_n$,
$C_n$, $D_n$, and $M_n$ in the appendix.
\end{remark}

The following conjecture appeared in \cite[Conjecture 2.10]{MPY}. 

\begin{conjecture}\label{mpy}
For all $n$, there exists an isomorphism of graded $\mathfrak{S}_n$-representations $D_n\cong M_n$.
\end{conjecture}

In this paper, we prove that Conjecture \ref{mpy} holds in low degree.

\begin{theorem}\label{mpy-low}
For all $i\leq 7$ and all $n$, there is an isomorphism of $\mathfrak{S}_n$-representations $D_n^i\cong M_n^i$.
\end{theorem}

Let $V$ be a graded representation of a finite group $\Gamma$.  We say that $V$ is {\bf strongly equivariantly log concave in degree \boldmath{$m$}}
if, for all $i\leq j\leq k\leq l$ with $j+k=i+l=m$, $V^i\otimes V^l$ is isomorphic to a subrepresentation of $V^j\otimes V^k$.  This may be rewritten
as the following sequence of inclusions:
$$V^0 \otimes V^m \subset V^1\otimes V^{m-1}\subset V^2 \otimes V^{m-2}\subset \cdots\subset\begin{cases}
V^{m/2}\otimes V^{m/2} &\text{if $m$ is even}\\
V^{(m-1)/2}\otimes V^{(m+1)/2} &\text{if $m$ is odd.}\end{cases}$$
We say that $V$ is {\bf strongly equivariantly log concave} if it is strongly equivariantly log concave in all degrees.\footnote{If $\Gamma$ is the trivial group, this is equivalent to the statement that the sequence of dimensions of the graded pieces of $V$ 
is log concave with no internal zeros.}

\begin{conjecture}\label{elc}
For all $n$, the graded $\mathfrak{S}_n$-representations $A_n$, $B_n$, $C_n$, $D_n$, and $M_n$ are all strongly equivariantly log concave.
\end{conjecture}

The conjecture for $A_n$ first appeared in \cite[Conjecture 5.3]{GPY}. 
In this paper, we prove that
Conjecture \ref{elc} holds in low degree.

\begin{theorem}\label{elc-low}
For all $n$, the graded $\mathfrak{S}_n$-representations $A_n$, $B_n$, $C_n$, and $D_n$ are all strongly equivariantly 
log concave in degrees $\leq 14$.  The graded $\mathfrak{S}_n$-representation $M_n$ is strongly equivariantly log concave in degrees $\leq 8$.
\end{theorem}

\begin{remark}
Conjecture \ref{elc} generalizes to equivariant log concavity statements about matroids, oriented matroids, and hyperplane arrangements
with symmetries, as we explain in Section \ref{sec:elc}.  If we consider the trivial symmetry group, all of these statements boil down to
the log concavity results of Adiprasito--Huh--Katz for (reduced) characteristic polynomials of matroids \cite{AHK} 
and Ardila--Denham--Huh for $h$-polynomials of broken-circuit complexes \cite{ADH}.
Conjecture \ref{elc} is what you get by considering the case of the
matroid, oriented matroid, or hyperplane arrangement associated with the roots of the Lie algebra $\mathfrak{sl}_n$, which are acted on
by the symmetric group $\mathfrak{S}_n$.
\end{remark}

Our approach to Theorem \ref{mpy-low} is to use the theory of representation stability, due to Church--Ellenberg--Farb \cite{CEF}.  Loosely speaking, if $V_n$ is a representation of $\mathfrak{S}_n$ for all $n$, there is a notion of what it means for the sequence
$V$ to {\bf stabilize at \boldmath{$d$}}.  If this happens, then for every $n\geq d$, $V_{n+1}$ can be computed algorithmically from $V_n$.
We show that $D^i$ and $M^i$ each stabilize at $3i$, which means that we can prove Theorem \ref{mpy-low} by checking 
that $D_n\cong M_n$ for all $n\leq 21$.
This type of argument is in some sense the primary motivation for the concept of representation stability, though we are unaware
of another situation in which the theory has been applied in such a direct way to prove that two infinite
sequences of representations of symmetric groups are isomorphic.

Our approach to Theorem \ref{elc-low} is similar, but with an additional subtlety.  
For this theorem, we need to understand when stability occurs, not just for the sequences $B^i$ and $D^i$,
but also for the sequences $B^j\otimes B^k$ and $D^j\otimes D^k$. (We show that the statements about $A$, $C$, and $M$ follow from 
the statements about $B$ and $D$.)
This requires a general statement about the stability range for the tensor product of two stable sequences (Theorem \ref{tensor}),
the proof of which relies on a powerful result about Kronecker coefficients \cite{BOR} 
that has not previously been incorporated into the literature on representation
stability.  

Ultimately, both theorems are proved by using representation stability to reduce to a finite number of cases that can be checked
on a computer.  We perform these checks using the software package SageMath \cite{sagemath}. 

\vspace{\baselineskip}
\noindent
{\em Acknowledgments:}
We are grateful to Vic Reiner for introducing us to the paper \cite{BOR}, to David Speyer
for communicating to us the proof of Proposition \ref{tensor-elc}, 
and to Ben Young for writing preliminary versions of the code that formed the basis for our
computer calculations.  We are also grateful to Ben Knudsen and Claudiu Raciu for conversations over the years regarding the two presentations
in Theorem \ref{two}.  Finally, we thank the referees for their insightful comments and corrections. 

\section{Equivariant log concavity}\label{sec:elc}
Recall that a sequence of non-negative integers $a_0,a_1,\ldots$ is called {\bf log concave} if $a_i^2\geq a_{i-1}a_{i+1}$ for all
$i$, and it is called {\bf log concave with no internal zeros} if there does not exist $i<j<k$ such that $a_j=0$ and $a_i\neq 0\neq a_k$.
This latter condition is equivalent to the statement that, whenever $i\leq j\leq k\leq l$ and $j+k=i+l$, we have $a_ia_l\leq a_ja_k$.
The notion of log concavity with no internal zeros has the advantage that it is preserved under convolution.
That is, if $a_0,a_1,\ldots$ and $b_0,b_1,\ldots$ are both log concave with no internal zeros, then the same is true for $c_0,c_1,\ldots$,
where $c_k = a_0b_k + a_1b_{k-1}+\cdots+a_kb_0$.

Let $V$ be a graded representation of a finite group $\Gamma$.
We say that $V$ is {\bf weakly equivariantly log concave} if, for all $i$, $V^{i-1}\otimes V^{i+1}$
is isomorphic to a subrepresentation of $V^i\otimes V^i$. 
We say that $V$ is {\bf strongly equivariantly log concave} if, whenever $i\leq j\leq k\leq l$ and $j+k=i+l$,
$V^i\otimes V^l$ is isomorphic to a subrepresentation of $V^j\otimes V^k$.  If $\Gamma$ is trivial, weak equivariant log concavity
is equivalent to log concavity of the sequence of dimensions, and strong equivariant log concavity is equivalent to log concavity
with no internal zeros of the sequence of dimensions.  However, this is not the case when $\Gamma$ is nontrivial, as the following
example illustrates.

\begin{example}\label{weak}
Let $\Gamma = S_2$, let $\tau$ denote the trivial representation of $\Gamma$, and let $\sigma$ denote the sign representation.
Define
$$V^i = \begin{cases}
\tau \oplus \sigma^{\oplus 3} &\text{if $i=0$ or $3$}\\
\tau^{\oplus 2} \oplus \sigma^{\oplus 2} &\text{if $i=1$ or $2$}\\
0 &\text{otherwise.}\end{cases}$$
Then $V$ is weakly equivariantly log concave and has no internal zeros, but $V^0\otimes V^3 \cong \tau^{\oplus 10}\oplus \sigma^{\oplus 6}$
is not isomorphic to a subrepresentation of $V^1\otimes V^2 \cong \tau^{\oplus 8}\oplus \sigma^{\oplus 8}$, so $V$ is not strongly equivariantly
log concave.
\end{example}

\begin{example}
Let $V$ be as in Example \ref{weak}, and let $W$ be the graded representation with $W^0=W^1=\tau$ and $W^i=0$ for all $i>1$.
Then both $V$ and $W$ are weakly equivariantly log concave with no internal zeros, but $V\otimes W$ fails to satisfy the weak equivariant log concavity condition when $i=2$.
Hence the property of weak equivariant log concavity with no internal zeros is not preserved under tensor product.
\end{example}

The claim that strong equivariant log concavity is the ``correct notion'' in the equivariant setting is justified by the following proposition, whose proof
was communicated to us by David Speyer.

\begin{proposition}\label{tensor-elc}
If $V$ and $W$ are strongly equivariantly log concave representations of $\Gamma$, then so is $V\otimes W$.
More generally, if $V$ and $W$ are both  strongly equivariantly log concave in degrees $\leq m$ (as defined in the introduction),
then so is $V\otimes W$.
\end{proposition}

\begin{proof}
Let $i\leq j\leq k\leq l$ be given with $j+k=i+l$.  
We have 
\begin{eqnarray*} (V\otimes W)^j \otimes (V\otimes W)^k &=& \bigoplus_{p,q} V^p\otimes W^{j-p} \otimes V^q\otimes W^{k-q}\\
&=& \bigoplus_{p,q} V^{j-l+q}\otimes W^{l-q} \otimes V^{k-i+p}\otimes W^{i-p},
\end{eqnarray*}
where the second line is obtained from the first by applying the affine transformation that takes $(p,q)$ to $(j-l+q,k-i+p)$.
Similarly, we have
\begin{eqnarray*} (V\otimes W)^i \otimes (V\otimes W)^l &=& \bigoplus_{p,q} V^p\otimes W^{i-p} \otimes V^q\otimes W^{l-q}\\
&=& \bigoplus_{p,q} V^{j-l+q}\otimes W^{k-q} \otimes V^{k-i+p}\otimes W^{j-p}.
\end{eqnarray*}
Working in the ring of virtual representations of $\Gamma$, consider the sum of the first two lines minus the sum of the last two lines in the 
previous two sentences.  We get
\begin{eqnarray*} && 2\Big((V\otimes W)^j \otimes (V\otimes W)^k - (V\otimes W)^i \otimes (V\otimes W)^l\Big)\\
&=& \sum_{p,q} V^p\otimes W^{j-p} \otimes V^q\otimes W^{k-q} + \sum_{p,q} V^{j-l+q}\otimes W^{l-q} \otimes V^{k-i+p}\otimes W^{i-p}\\
&& - \sum_{p,q} V^p\otimes W^{i-p} \otimes V^q\otimes W^{l-q} - \sum_{p,q} V^{j-l+q}\otimes W^{k-q} \otimes V^{k-i+p}\otimes W^{j-p}\\
&=& \sum_{p,q}\Big(V^p\otimes V^q - V^{j-l+q}\otimes V^{k-i+p}\Big)\otimes \Big(W^{j-p}\otimes W^{k-q} - W^{l-q}\otimes W^{i-p}\Big).
\end{eqnarray*}
By strong equivariant log concavity of $V$, $V^p\otimes V^q - V^{j-l+q}\otimes V^{k-i+p}$ is the class of an honest representation
if $p\geq j-l+q$, and otherwise it is minus the class of an honest representation.
Similarly, by strong equivariant log concavity of $W$, $W^{j-p}\otimes W^{k-q} - W^{l-q}\otimes W^{i-p}$
is the class of an honest representation if $j-p\leq l-q$, and otherwise it is minus the class of an honest representation.
Since $p\geq j-l+q$ if and only if $j-p\leq l-q$, the tensor product
$$\Big(V^p\otimes V^q - V^{j-l+q}\otimes V^{k-i+p}\Big)\otimes \Big(W^{j-p}\otimes W^{k-q} - W^{l-q}\otimes W^{i-p}\Big)$$
is always equal to the class of an honest representation.

In general, any class in the virtual representation ring of $\Gamma$ that is equal to half the class of an honest representation
is itself the class of an honest representation.  Thus
$$(V\otimes W)^j \otimes (V\otimes W)^k - (V\otimes W)^i \otimes (V\otimes W)^l$$ 
is the class of an honest representation, which is equivalent to the statement that $(V\otimes W)^i \otimes (V\otimes W)^l$
is isomorphic to a subrepresentation of $(V\otimes W)^j \otimes (V\otimes W)^k$.

Finally, we need to check that, if $j+k=m$, then we only need to assume that $V$ and $W$ are strongly equivariantly log concave
in degrees $\leq m$.  When we used strong equivariant log concavity of $W$, we used it in degree $m-p-q\leq m$.
When we used strong equivariant log concavity of $V$, we used it in degree $p+q$.  If $p+q>m$, then
we have either $p>j$ or $q>k$, and we also have either $p>i$ or $q>l$.  This implies that
the factor $W^{j-p}\otimes W^{k-q} - W^{l-q}\otimes W^{i-p}$ is equal to zero, and we can therefore ignore that term of the sum.
\end{proof}

\begin{remark}
The definition of strong equivariant log concavity can be generalized by replacing the virtual representation ring of a finite group
with any partially ordered ring.  More precisely, there should be a subset of ``non-negative elements'' (analogous to honest representations)
that includes $0$ and is closed under addition and multiplication, 
and we define a sequence $a_0,a_1,\ldots$ to be {\bf strongly log concave} if, whenever
we have $i\leq j\leq k\leq l$ with $j+k=i+l$, $a_ja_k-a_ia_l$ is non-negative.  Proposition \ref{tensor-elc} generalizes to say that,
if our ring has the added property that $x$ is non-negative whenever $2x$ is non-negative, 
then the convolution of strongly log concave sequences is again strongly log concave.
\end{remark}

We now make a number of conjectures that generalize Conjecture \ref{elc}.
Let $M$ be a matroid of positive rank, 
and let $\Gamma$ be a group acting on the ground set of $M$ preserving the collection of independent sets.
The {\bf Orlik--Solomon algebra} $A(M)$ is defined as a quotient of the exterior algebra with generators
indexed by the ground set of the matroid \cite{OS}, and the {\bf reduced Orlik--Solomon algebra} $B(M)$ is the subalgebra of $A(M)$
generated by differences of the generators.

\begin{conjecture}\label{matroid}
The Orlik--Solomon algebra $A(M)$ and 
the reduced Orlik--Solomon algebra $B(M)$
are strongly equivariantly log concave.
\end{conjecture}

\begin{remark}
When $M$ is the braid matroid of rank $n-1$,
$A(M)$ is isomorphic to $A_n$ and $B(M)$ is isomorphic to $B_n$.
More generally, if $M$ is the matroid associated with a finite set of hyperplanes in a complex vector space,
then $A(M)$ is isomorphic to the cohomology ring of the complement of the hyperplanes and $B(M)$
is isomorphic to the cohomology ring of the projectivized complement \cite{OS}.
\end{remark}

\begin{remark}
We always have a canonical isomorphism \cite[Proposition 2.18]{YuzOS}
\begin{equation}\label{AB}A(M)\cong B(M)\otimes\mathcal{E}_\Q[t],\end{equation}
where $\mathcal{E}_\Q[t]$ is the exterior algebra on the single variable $t$.  Hence 
strong equivariant log concavity of $B(M)$ implies
strong equivariant log concavity of $A(M)$
by Proposition \ref{tensor-elc}.
\end{remark}

\begin{remark}  
The dimensions of the graded pieces of $A(M)$ and $B(M)$ are the coefficients of the characteristic
polynomial and the reduced characteristic polynomial of $M$, respectively.  Thus, when $\Gamma$ is the trivial group,
Conjecture \ref{matroid} specializes to the main theorem of
Adiprasito--Huh--Katz \cite[Theorem 9.9]{AHK}. 
\end{remark}


\begin{remark}
The conjecture that $A(M)$ is strongly equivariantly log concave
originally appeared in \cite[Conjecture 5.3 and Remark 5.8]{GPY}, 
where it was proved for uniform matroids with $\Gamma$ equal to the full group of permutations of the ground set \cite[Proposition 5.7
and Remark 5.8]{GPY}.  The argument there can easily be adapted to prove strong equivariant log concavity of $B(M)$ for uniform matroids as well.
\end{remark}

Let $M$ be an oriented matroid, and let $\Gamma$ be a group acting on the ground set of $M$ preserving the collection of signed
circuits.  The {\bf Cordovil algebra} $C(M)$ is defined as a quotient of the polynomial ring with generators indexed by the elements of the ground
set of $M$ \cite{Co}.

\begin{conjecture}\label{oriented}
The Cordovil algebra $C(M)$ is a strongly equivariantly log concave graded representation of $\Gamma$.
\end{conjecture}

\begin{remark}
When $M$ is the oriented braid matroid of rank $n-1$, the Cordovil algebra $C(M)$ is isomorphic to $C_n$.
More generally, if $M$ is the oriented matroid associated with a finite set $\cA$ of hyperplanes in a real vector
space $V$, then $C(M)$ is isomorphic to the cohomology of the space $$V\otimes\R^3\setminus\bigcup_{H\in \cA} H\otimes\R^3,$$
with degrees halved; see \cite[Corollary 5.6]{dLS} or \cite[Example 5.8]{Moseley}.
\end{remark}

\begin{remark}
As with the Orlik--Solomon algebra, the dimensions of the graded pieces of the Cordovil algebra are the coefficients of the characteristic polynomial
of the underlying matroid \cite[Corollary 2.8]{Co}.  This means that, in the case where $\Gamma$ is the trivial group,
Conjecture \ref{oriented} follows from \cite{AHK}.
\end{remark}

Let $\cA$ be a finite set of hyperplanes in a vector space $V$, equipped with a linear action of $\Gamma$
that preserves the hyperplanes.  The {\bf Orlik--Terao algebra} $OT(\cA)$
is defined as the subalgebra of rational functions on $V$ generated by the reciprocals of the linear functions that vanish
on the hyperplanes.  This algebra is Cohen--Macaulay, and it comes equipped with a canonical linear system of parameters \cite[Proposition 7]{PS}.
We denote by $M(\cA)$ the quotient of $OT(\cA)$ by this linear system of parameters.  The {\bf Artinian Orlik--Terao algebra} $AOT(\cA)$
is defined as the quotient of $OT(\cA)$ by the squares of the generators.

\begin{conjecture}\label{arrangement}
The algebras $M(\cA)$ and $AOT(\cA)$ are strongly equivariantly log concave graded representations of $\Gamma$.
\end{conjecture}

\begin{remark}
When $\cA$ is the Coxeter arrangement associated with $\mathfrak{sl}_n$, $M(\cA) \cong M_n$.
More generally, when $V$ is a vector space over the rational numbers, $OT(\cA)$ is isomorphic to the torus equivariant
intersection cohomology of the hypertoric variety associated with $\cA$, and $M(\cA)$ is isomorphic to the ordinary intersection cohomology
(both with degrees halved) \cite[Corollary 4.5]{TP08}.  We note that intersection cohomology does not usually come equipped
with a ring structure; the fact that it does in this case is a special feature of hypertoric varieties.
\end{remark}

\begin{remark}\label{AOTC}
Suppose that $V$ is a vector space over the real numbers and $M$ is the oriented matroid associated with $\cA$.
The Artinian Orlik--Terao algebra $AOT(\cA)$ and the Cordovil algebra $C(M)$
are typically not isomorphic as rings, but they are isomorphic as graded representations of $\Gamma$.\footnote{We thank Matt Douglass, G\"otz Pfeiffer, Vic Reiner, and Gerhard R\"ohrle for informing us of this fact and outlining the proof.}
If in addition $\cA$ is unimodular, meaning that the hyperplanes have rational slope with respect to some lattice and the subgroup generated
by any subset of the primitive normal vectors is saturated in that lattice, then there is a canonical graded ring isomorphism $AOT(\cA) \cong C(M)$.
The Coxeter arrangement associated with $\mathfrak{sl}_n$ has this property, thus its Artinian Orlik--Terao algebra is isomorphic to $C_n$.
\end{remark}

\begin{remark}
As with the Orlik--Solomon algebra and the Cordovil algebra, the dimensions of the graded pieces of the Artinian
Orlik--Terao algebra are the coefficients of the characteristic polynomial of the associated matroid.
On the other hand, the dimensions of the graded pieces of $M(\cA)$ are the coefficients of the $h$-polynomial of the broken circuit complex
of the associated matroid \cite[Proposition 7]{PS}, which is known to form a log concave sequence with no internal zeros 
by \cite[Theorem 1.4]{ADH}.
Thus Conjecture \ref{arrangement} holds when $\Gamma$ is trivial.
\end{remark}

\section{Representation stability}
Let $\operatorname{C}$ be a category.  We will refer to a functor from $\operatorname{C}$ to $\Vec_\Q$ 
as a {\bf \boldmath{$\operatorname{C}$}-module}.
The three main categories that we will discuss are the category $\FB$ of finite sets with bijections, the category $\FI$
of finite sets with injections, and the category $\FI^\#$ of finite sets with partially defined injections.
Given an $\FB$-module $P$ and a natural number $n$, we obtain a representation $P_n := P([n])$ of the symmetric group $\mathfrak{S}_n$,
and $P$ is determined up to isomorphism by the collection $\{P_n\mid n\in\N\}$.

Given a partition $\la$, we write $V_\la$ to denote the corresponding irreducible representation of $\mathfrak{S}_{|\la|}$.
Given an integer $n\geq |\la| + \la_1$, we write
$$\la[n] := (n-|\la|,\la_1,\ldots,\la_l)$$ for the partition of $n$ obtained from $\la$ by adding a new first part of size $n-|\la|$.
For any $d$, let $$\Lambda_d:= \{\la\mid |\la| + \la_1\leq d\}.$$
Given an $\FB$-module $P$ and a positive integer $d$, we say that
{\bf \boldmath{$P$} stabilizes at \boldmath{$d$}} if there exists a collection of natural numbers
$\{r_\la\mid \la\in\Lambda_d\}$ such that,
for all $n\geq d$, there is an isomorphism of $\mathfrak{S}_n$-representations\footnote{This terminology does not
imply sharpness.  Any $\FB$-module that stabilizes at $d$ also stabilizes at $e$ for all $e\geq d$.}
$$P_n \cong \bigoplus_{\la\in\Lambda_d} V_{\la[n]}^{\oplus r_\la}.
$$
If an $\FB$-module $P$ stabilizes at $d$ for some $d$, then we will say that $P$ is {\bf stable}.  The following two lemmas are straightforward.

\begin{lemma}\label{sums}
Suppose that $P$ and $Q$ are $\FB$-modules.  If any two of the modules $P$, $Q$, and $P\oplus Q$ stabilize at $d$,
so does the third.
\end{lemma}

\begin{lemma}\label{enough}
Suppose that $P$ and $Q$ are $\FB$-modules and that $P$ and $Q$ both stabilize at $d$.
\begin{enumerate}
\item If $P_n\cong Q_n$ for all $n\leq d$, then $P\cong Q$.
\item If $P_n$ is isomorphic to a subrepresentation of $Q_n$ for all $n\leq d$, then $P$ is isomorphic to a submodule of $Q$.
\end{enumerate}
\end{lemma}

The following theorem, which is proved using slightly different language in \cite[Theorem 1.2]{BOR}, is not at all straightforward.

\begin{theorem}\label{tensor}
Suppose that $P$ and $Q$ are $\FB$-modules such that $P$ stabilizes at $d$ and $Q$ stabilizes at $e$.
Then $P\otimes Q$ stabilizes at $d+e$.
\end{theorem}

\begin{proof}
Write $\{r_\la\mid \la\in\Lambda_d\}$ and $\{s_\mu\mid\mu\in\Lambda_e\}$ for the multiplicities associated with $P$ and $Q$.
For all $n\geq \max\{d,e\}$, we have
$$(P\otimes Q)_n  =  P_n\otimes Q_n \cong \bigoplus_{\substack{\la\in\Lambda_d\\ \mu\in\Lambda_e}} V_{\la[n]}^{\oplus r_\la}\otimes V_{\mu[n]}^{\oplus s_\mu} \cong \bigoplus_{\substack{\la\in\Lambda_d\\ \mu\in\Lambda_e}} \left(V_{\la[n]}\otimes V_{\mu[n]}\right)^{\oplus r_\la s_\mu}.$$
By \cite[Theorem 1.2]{BOR}, the $\FB$-module that sends $n$ to $V_{\la[n]}\otimes V_{\mu[n]}$ for $n\geq \max\{|\la|+\la_1,|\mu|+\mu_1\}$
and to $0$ otherwise stabilizes at $|\la| + \la_1 + |\mu| + \mu_1 \leq d + e$.  
Since this is true for all $\la\in\Lambda_d$ and $\mu\in\Lambda_e$, the result follows.
\end{proof}

There is a unique $\FI$-module $P(\la)$ such that, for any $\FI$-module $Q$, we have
$$\Hom_{\FI\text{\!-mod}}(P(\la), Q) = \Hom_{S_{|\la|}\text{-mod}}(V_\la, Q_{|\la|}).$$
An $\FI$-module is called {\bf free} if it is isomorphic to a direct sum of $\FI$-modules of the form $P(\la)$.
An $\FI$-module is called {\bf finitely generated} if it is isomorphic to a quotient of a free $\FI$-module
with finitely many summands.

\begin{remark}\label{FIfg}
The central observation of Church--Ellenberg--Farb is that an $\FB$-module stabilizes if and only if 
the $\FB$-module structure admits an extension to a finitely generated $\FI$-module structure \cite[Theorem 1.13]{CEF}.
Since tensor products of finitely generated $\FI$-modules are again finitely generated \cite[Proposition 2.3.6]{CEF},
this observation immediately implies that the tensor product of two stable $\FB$-modules is again stable.
However, the statement that the point at which stabilization occurs is weakly sub-additive under tensor product (Theorem \ref{tensor}) 
is not at all clear  from the representation stability literature, and relies instead on the work of Briand--Orellana--Rosas.  This result
is sharper, for example, than the one that one obtains from \cite[Proposition 2.23]{KupersMiller}.
\end{remark}

\begin{remark}
Theorem \ref{tensor} is particularly interesting when $P$ and $Q$ are restrictions to $\FB$ of free $\FI$-modules.
For all $n\leq |\la|$, we have $$P(\la)_n \cong \Ind_{S_\la\times S_{n-|\la|}}^{\mathfrak{S}_n}\left(V_\la\right)
\cong \bigoplus_{\mu\in\Lambda_\la} V_{\mu[n]},$$
where $\Lambda_\la := \{\mu\mid \la_i\geq \mu_i\geq \la_{i+1}\;\text{for all $i$}\}$, and $S_{n-|\la|}$ acts trivially on $V_\la$.  
Since $|\mu|\leq |\la|$ and $\mu_1\leq \la_1$ for all $\mu\in\Lambda_\la$, with equality when $\mu=\la$,
this implies that $P(\la)$ stabilizes sharply at $|\la|+\la_1$.

Church--Ellenberg--Farb prove that tensor products of free modules are free, and we can therefore write \cite[Equation (17)]{CEF}
$$P(\la)\otimes P(\mu) \cong \bigoplus_\nu P(\nu)^{\oplus d_{\la\mu}^{\nu}}.$$
Theorem \ref{tensor} implies that \begin{equation}\label{d}|\nu|+\nu_1 \leq |\la|+\la_1 + |\mu|+\mu_1\end{equation} 
whenever $d_{\la\mu}^{\nu} \neq 0$.

When $|\la| = |\mu| = |\nu|$, $d_{\la\mu}^{\nu}$ is the Kronecker coefficient that measures the multiplicity of $V_\nu$ in $V_\la\otimes V_\mu$,
and Equation \eqref{d} is trivial.
When $|\la|+|\mu| = |\nu|$, $d_{\la\mu}^{\nu}$ is the Littlewood--Richardson coefficient that measures the multiplicity of $V_\nu$ in 
$\Ind_{S_{|\la|}\times S_{|\mu|}}^{S_{|\nu|}}\Big(V_\la\boxtimes V_\mu\Big)$, and Equation \eqref{d}
follows from the interpretation of $d_{\la\mu}^{\nu}$ in terms of skew tableaux.  
For general $\la$, $\mu$, and $\nu$, we believe that Equation \eqref{d} was not previously known.
\end{remark}

\begin{example}\label{AC}
Let $A^i$ and $C^i$ be the $\FB$-modules that take $n$ to $A^i_n$ and $C^i_n$, respectively.
Hersh and Reiner \cite[Theorem 1.1]{HershReiner} prove that $A^i$ stabilizes at $3i+1$ and $C^i$ stabilizes at $3i$.
Both extend to free $\FI$-modules, so this is equivalent to the statement that, for each summand $P(\la)$ of $A^i$ (respectively $C^i$),
$|\la|+\la_1$ is less than or equal to $3i+1$ (respectively $3i$).
\end{example}

\begin{example}\label{B}
Let $B^i$ be the $\FB$-module that take $n$ to $B^i_n$.
Equation \eqref{AB} says that $A_n \cong B_n\otimes \mathcal{E}_\Q[t]$, and therefore $A^i \cong B^i \oplus B^{i-1}$.
Thus $B^i$ stabilizes at $3i+1$ by an inductive argument involving Lemma \ref{sums} and Example \ref{AC}.
\end{example}

\begin{example}\label{D}
Let $W_n := W_n^0 \oplus W_n^1$,
where, $W_n^0 = V_{[n]}$ is the $1$-dimensional trivial representation of $\mathfrak{S}_n$, and $W_n^1 = V_{[n-1,1]}$ is the standard representation.  There exists an isomorphism of graded $\mathfrak{S}_n$-representations \cite[Proposition 2.5]{MPY}
\begin{equation}\label{CD}C_n \cong D_n \otimes W_n.\end{equation}
Let $D^i$ be the $\FB$-module that takes $n$ to $D^i_n$, and let $W^1$ be the $\FB$-module that takes $n$ to $W^1_n \cong V_{[n-1,1]}$.
Equation \eqref{CD} gives
$C^i \cong D^i \oplus (D^{i-1} \otimes W^1).$
Note that $W^1$ stabilizes at $2$, thus $D^i$ stabilizes at $3i$ by an inductive argument involving Lemma \ref{sums}, Theorem \ref{tensor}, 
and Example \ref{AC}.
\end{example}

The following proposition follows from the deep result \cite[Theorem 4.1.5]{CEF}, which provides an equivalence between the category of
$\FI^\#$-modules and the category of $\FB$-modules.

\begin{proposition}\label{free}
Suppose that $P$ is an $\FI^\#$-module that stabilizes at $d$.  Then any $\FI^\#$-submodule or $\FI^\#$-quotient module
of $P$ also stabilizes at $d$.
\end{proposition}

\begin{proof}
By \cite[Theorem 4.1.5]{CEF}, any $\FI^\#$-module is free as an $\FI$-module.  Since $P$ stabilizes at $d$, this means that
we have a collection of natural numbers $\{r_\la\mid\la\in\Lambda_d\}$
and an isomorphism $$P \cong \bigoplus_{\la\in\Lambda_d} P(\la)^{\oplus r_\la}.$$
Furthermore, the result \cite[Theorem 4.1.5]{CEF} implies that
any $\FI^\#$-submodule or $\FI^\#$-quotient module of $P$
is isomorphic to $$\bigoplus_{\la\in\Lambda_d} P(\la)^{\oplus s_\la}$$
for some collection of multiplicities $\{s_\la\mid\la\in\Lambda_d\}$ with $s_\la\leq r_\la$ for all $\la\in\Lambda_d$.
In particular, such a submodule or quotient module stabilizes at $d$.
\end{proof}

\begin{example}\label{T}
Let $T_n = \Q[z_1,\ldots,z_n]$ be the polynomial ring in $n$ variables, and let $T^i$ be the $\FB$-module that takes $n$ to $T_n^i$.
The $\FB$-module structure on $T^i$ extends canonically to an $\FI^\#$-module structure, where a partially defined inclusion $\varphi$
from $[m]$ to $[n]$ sends $z_i$ to $z_{\varphi(i)}$ if $\varphi(i)$ is defined and to $0$ otherwise.
The module $T^1$ is the free module associated with the partition $[1]$, and therefore stabilizes at $2$.
Since $T^i$ is an $\FI^\#$-quotient module of the tensor power $(T^1)^{\otimes i}$, 
Theorem \ref{tensor} and Proposition \ref{free} together imply that $T^i$ stabilizes at $2i$.
\end{example}

\begin{example}\label{R}
Let $R_n = \Q[z_1,\ldots,z_n]/\langle z_1+\cdots +z_n\rangle$, and let 
$R^i$ be the $\FB$-module that takes $n$ to $R^i_n$.  
Consider the polynomial ring $\Q[t]$ with trivial $S_n$-action, along with the
$S_n$-equivariant isomorphism $T_n \cong R_n \otimes \Q[t]$
that sends $z_i$ to $z_i\otimes 1 + 1\otimes t$.
This isomorphism shows that \begin{equation}\label{OKreferee}T^i \cong \bigoplus_{0\leq j\leq i} R^j.\end{equation}
We claim that $R^i$ stabilizes at $2i$.  Indeed, if we assume that it holds for all $j<i$, then it holds for $i$ by 
Lemma \ref{sums}, Example \ref{T}, and Equation \eqref{OKreferee}.
\end{example}

\begin{example}\label{OT}
Let $OT_n$ be the Orlik--Terao algebra of the Coxeter arrangement associated with $\mathfrak{sl}_n$.  By definition, this is the subalgebra
of rational functions in the variables $y_1,\ldots,y_n$ generated by the functions $x_{jk} = \frac{1}{y_j-y_k}$.  This ring admits a grading
with $\deg x_{jk} = 1$, and we let $OT^i$ be the $\FB$-module that takes $n$ to $OT^i_n$.
One can see from the explicit presentation in Theorem \ref{OT-pres} that 
the $\FB$-module structure on $OT^i$ extends canonically to an $\FI^\#$-module structure, where a partially defined inclusion $\varphi$
from $[m]$ to $[n]$ sends $x_{jk}$ to $x_{\varphi(j)\varphi(k)}$ if $\varphi(j)$ and $\varphi(k)$ are both defined and to $0$ otherwise.
Remark \ref{AOTC} tells us that $OT^1 \cong C^1$, which stabilizes at $3$ by Example \ref{AC}.  
Since $OT^i$ is an $\FI^\#$-quotient module of the tensor power $(OT^1)^{\otimes i}$, Theorem \ref{tensor} and Proposition \ref{free}
together imply that $OT^i$ stabilizes at $3i$.
\end{example}

\begin{example}\label{M}
Let $M^i$ be the $\FB$-module that takes $n$ to $M^i_n$.  
There exists an isomorphism of graded $\mathfrak{S}_n$-representations \cite[Section 2.1]{MPY}
\begin{equation}\label{MOT}OT_n \cong R_n \otimes M_n,\end{equation} and therefore
$$OT^i \cong \bigoplus_{j+k=i} R^j \otimes M^k.$$
Thus $M^i$ stabilizes at $3i$ by an inductive argument involving
Lemma \ref{sums}, Theorem \ref{tensor}, and Examples \ref{R} and \ref{OT}.
\end{example}

\begin{remark}
The isomorphisms in Equations \eqref{CD} and \eqref{MOT} are not canonical, nor are they isomorphisms of algebras.
Indeed, each one is proved by constructing a spectral sequence that degenerates because all cohomology groups involved vanish in odd degrees.
\end{remark}

\begin{remark}
The $\FB$-modules $R^i$ and $M^i$ do not extend canonically to $\FI$-modules (rather, they extend canonically to 
$\FI^{\operatorname{op}}$-modules).  On the other hand, Remark \ref{FIfg} along with Examples \ref{R} and \ref{M}
together imply that both $R^i$ and $M^i$ do admit (perhaps noncanonical) extensions to $\FI$-modules.
In the case of $R^i$, this can be seen by identifying $R_n$ with the subalgebra of $T_n$ generated by the elements $z_j-z_k$.
In the case of $M^i$, it can be regarded as evidence for Conjecture~\ref{mpy}.
\end{remark}

\begin{remark}
Conjecture \ref{elc} says that, for all $i\leq j\leq k\leq l$ with $j+k=i+l$, $A^i\otimes A^l$ is isomorphic to a sub-$\FB$-module of $A^j\otimes A^k$,
or equivalently a quotient $\FB$-module, since the category of $\FB$-modules is semisimple.
A much stronger conjecture would be that $A^i\otimes A^l$ is isomorphic to a quotient $\FI$-module of $A^j\otimes A^k$.
This is indeed true when $i=0$, as the multiplication map
$$A^j\otimes A^k\to A^{j+k} \cong A^0\otimes A^{j+k}$$
is surjective.  However, the fact that we know of no natural map from
$A^j\otimes A^k$ to $A^i\otimes A^l$ when $i>0$ leads us to doubt that this stronger conjecture holds.
The same remark applies with $A$ replaced by $B$, $C$, or $D$, or by $M$ in the category of
$\FI^{\operatorname{op}}$-modules.
\end{remark}

\section{Proofs}

In this section, we describe our computer assisted proofs of Theorems \ref{mpy-low} and \ref{elc-low}.
For computational purposes, one can explicitly obtain $A_n^i$ and $C_n^i$ using 
\cite[Equations (25) and (26), Theorem 2.7, and Section 2.7]{HershReiner}, and then obtain $B_n^i$ and $D_n^i$
using Equations \eqref{AB} and \eqref{CD}.
Our calculations of the representation $M_n$ rely on the recursive formula \cite[Theorem 3.2]{MPY}, 
which is derived using a canonical stratification of the hypertoric variety $X_n$.  This calculation is much more computationally
intensive than the ones used to compute $A_n$, $B_n$, $C_n$, and $D_n$, which is why the statement of Theorem \ref{elc-low}
is weaker for $M_n$ than for the other four graded representations.
The computer code used in this paper can be found at \url{https://github.com/jacobmatherne/ELCandRS}.

\begin{proof}[Proof of Theorem \ref{mpy-low}]
By Lemma \ref{enough}(1), Example \ref{D}, and Example \ref{M}, it is sufficient to check that $D_n\cong M_n$ for all $n\leq 21$.
We have performed these checks using SageMath.
\end{proof}

\begin{remark}
In fact, we checked Conjecture \ref{mpy} for all $n\leq 22$.  Thus the first unknown statement of Conjecture \ref{mpy}
is that $D_{23}^8$ is isomorphic to $M_{23}^8$.
Conjecture \ref{mpy} had previously only been checked for all $n\leq 10$ \cite[Remark 2.11]{MPY}.
\end{remark}

\begin{proof}[Proof of Theorem \ref{elc-low}]
We begin with $B_n$.  We need to show that, for all $i\leq j\leq k\leq l$ with $j+k=i+l=m\leq 14$, $B_n^i\otimes B_n^l$ is isomorphic
to a subrepresentation of $B_n^j\otimes B_n^k$.  By Theorem \ref{tensor} and Example \ref{B}, both $B^i\otimes B^l$ and
$B^j\otimes B^k$ stabilize at $3m+2$.  Thus, by Lemma \ref{enough}(2), it is sufficient to check that 
$B_n^i\otimes B_n^l$ is isomorphic to a subrepresentation of $B_n^j\otimes B_n^k$ for all $n\leq 3m+2$.
The situation for $D_n$ is identical, except this time Example \ref{D} tells us that stabilization occurs at $3m$ rather than $3m+2$.
We have performed these checks for $B_n$ and $D_n$ using SageMath.

The statement for $A_n$ follows from the statement for $B_n$ using Equation \eqref{AB} and Proposition \ref{tensor-elc}.
The statement for $C_n$ follows from the statement for $D_n$ using Equation \eqref{CD} and Proposition \ref{tensor-elc}.

Finally, the statement for $M_n$ nearly follows from the statement for $D_n$ using Theorem \ref{mpy-low}.  The one part
that does not follow is the assertion that $M_n^0\otimes M_n^8$ is isomorphic to a subrepresentation of $M_n^1\otimes M_n^7$, since we do not
know that $M_n^8$ is isomorphic to $D_n^8$.  However, $M_n$ is generated in degree $1$ by Theorem \ref{M-pres}, hence we have
a surjection
$$M_n^1\otimes M_n^7\to M_n^8\cong M_n^0\otimes M_n^8.$$
This tells us that $M_n^0\otimes M_n^8$ is isomorphic to a quotient of $M_n^1\otimes M_n^7$, and therefore also a subrepresentation
by semisimplicity of the category of representations of $\mathfrak{S}_n$.
\end{proof}

\begin{remark}
By Theorem \ref{elc-low}, the weak equivariant log concavity statement that $A_n^{i-1}\otimes A_n^{i+1}$ is isomorphic to a subrepresentation
of $A_n^i\otimes A_n^i$ holds for all $i\leq 7$, and similarly for $B_n$, $C_n$, and $D_n$.  Combining this result with Theorem \ref{mpy-low},
the statement  that $M^{i-1}\otimes M^{i+1}$ is isomorphic to a subrepresentation
of $M^i\otimes M^i$ holds for all $i\leq 6$.
\end{remark}

\appendix
\section{Appendix:  Presentations}
In this section, we give explicit presentations of each of the rings that we consider in this paper.
Most of the results in this appendix are well known, with the exception of Theorem \ref{two}.  Theorem \ref{two}
can be deduced from the proof of \cite[Theorem 3]{EarlyReiner}, but we include a proof here for completeness.

We begin with the ring $A_n$, which was first computed by Arnol'd \cite{arnold}.

\begin{theorem}\label{OSA}
There exists an isomorphism $A_n \cong \mathcal{E}_\Q[x_{ij}]/\mathcal{I}^A_n$, where $\mathcal{E}_\Q[x_{ij}]$
is the exterior algebra with generators $x_{ij}$ for all distinct $i, j\in[n]$ and $\mathcal{I}^A_n$ is the ideal generated by the following
families of relations:
\begin{itemize}
\item $x_{ij} - x_{ji}$ for all $i, j$ distinct
\item $x_{ij}x_{jk} + x_{jk}x_{ki} + x_{ki}x_{ij}$ for all $i,j,k$ distinct.
\end{itemize}
The group $\mathfrak{S}_n$ acts by permuting the indices.
\end{theorem}

The isomorphism inverse to that of Equation \eqref{AB} sends $t$ to $\sum_{i\neq j} x_{ij}$, thus Theorem \ref{OSA} has the following 
corollary.

\begin{corollary}\label{B-pres}
There exists an $\mathfrak{S}_n$-equivariant isomorphism $B_n \cong \mathcal{E}_\Q[x_{ij}]/\mathcal{I}^B_n$, where
$$\mathcal{I}^B_n := \mathcal{I}^A_n + \left\langle \sum_{i\neq j} x_{ij}\right\rangle.$$
\end{corollary}

The ring $C_n$ was first computed by Cohen \cite{Cohen}; see alternatively \cite[Corollary 5.6]{dLS}.

\begin{theorem}\label{C-pres}
There exists an isomorphism $C_n \cong \Q[x_{ij}]/\mathcal{I}^C_n$, where $\Q[x_{ij}]$
is the polynomial ring with generators $x_{ij}$ for all distinct $i, j\in[n]$ and $\mathcal{I}^C_n$ is the ideal generated by the following
families of relations:
\begin{itemize}
\item $x_{ij} + x_{ji}$ for all $i, j$ distinct
\item $x_{ij}^2$ for all $i,j$ distinct
\item $x_{ij}x_{jk} + x_{jk}x_{ki} + x_{ki}x_{ij}$ for all $i,j,k$ distinct.
\end{itemize}
The group $\mathfrak{S}_n$ acts by permuting the indices.
\end{theorem}

We next give two equivariant presentations for $D_n$, neither of which has appeared before.

\begin{theorem}\label{two}
There exists an isomorphism $D_n\cong \Q[h_{ijk}]/\mathcal{I}^D_n$, where $\Q[h_{ijk}]$
is the polynomial ring with generators $h_{ijk}$ for distinct triples $i,j,k\in[n]$ and $\mathcal{I}^D_n$ is the ideal generated by the following
families of relations:
\begin{itemize}
\item $h_{ijk}+h_{jik}$ and $h_{ijk}+h_{ikj}$ for all $i,j,k$ distinct
\item $h_{ijk}-h_{ijl}+h_{ikl}-h_{jkl}$ for all $i,j,k,l$ distinct
\item $h_{ijk}^2$ for all $i,j,k$ distinct.
\end{itemize}
There also exists an isomorphism $D_n \cong \Q[x_{ij}]/\mathcal{J}^D_n$, where $\Q[x_{ij}]$
is the polynomial ring with generators $x_{ij}$ for all distinct $i, j\in[n]$ and $\mathcal{J}^D_n$ is the ideal generated by the following
families of relations:
\begin{itemize}
\item $x_{ij} + x_{ji}$ for all $i,j$ distinct
\item $\sum_{j\neq i} x_{ij}$ for all $i$
\item $(x_{ij} + x_{jk} + x_{ki})^2$ for all $i,j,k$ distinct.
\end{itemize}
In both cases, the group $\mathfrak{S}_n$ acts by permuting the indices.
\end{theorem}

\begin{proof}
The space $\Conf(n,U_1)/U_1$ is a disjoint union of contractible subspaces indexed by cyclic orderings
of the set $[n]$.  For $i,j,k\in[n]$ distinct, consider the {\bf cyclic Heaviside function} $h_{ijk}$ that takes
the value $1$ on those components where $i$, $j$, and $k$ appear in a counterclockwise order, and $0$ on those
components where they appear in a clockwise order.  These functions generate the ring of locally constant functions on
$\Conf(n,U_1)/U_1$ and satisfy the following families of relations:
\begin{itemize}
\item $h_{ijk}+h_{jik} = 1 = h_{ijk}+h_{ikj}$ for all $i,j,k$ distinct
\item $h_{ijk}-h_{ijl}+h_{ikl}-h_{jkl} = 0$ for all $i,j,k,l$ distinct
\item $h_{ijk}^2 = h_{ijk}$ for all $i,j,k$ distinct.
\end{itemize}
We consider the filtration of the ring of locally constant functions on $\Conf(n,U_1)/U_1$
for which the $p^\text{th}$ filtered piece is the space of functions that can be expressed as polynomials of degree at
most $p$ in the cyclic Heaviside functions.
Using the fact that $\Conf(n,U_1)/U_1$ is homeomorphic to the fixed point set of the action of $U_1$ on $\Conf(n,SU_2)/SU_2$
by right translation,
one can show that the graded ring $D_n$ is $\mathfrak{S}_n$-equivariantly isomorphic to the associated graded of the ring of 
locally constant functions on $\Conf(n,U_1)/U_1$ with respect to the cyclic Heaviside filtration 
\cite[Remark 2.9]{MPY}. 
Passing to the associated graded turns the relations above into the generators of $\mathcal{I}^D_n$,
and we obtain
an $\mathfrak{S}_n$-equivariant surjective map $$\Q[h_{ijk}]/\mathcal{I}^D_n\to D_n.$$
To see that it is an isomorphism, we break symmetry and make use of the isomorphism $D_n\cong C_{n-1}$
of Remark \ref{identify} by reducing the problem to checking that the composition
$$\Q[h_{ijk}]/\mathcal{I}^D_n\to D_n\to C_{n-1}$$ is an isomorphism.  

The first and second families of generators of $\mathcal{I}^D_n$ imply that the degree $1$ part of $\Q[h_{ijk}]/\mathcal{I}^D_n$
is spanned by the generators $\{h_{ijn}\mid i\neq j\in [n-1]\}$, subject to the relations $h_{ijn}+h_{jin} = 0$.  The map from
$\Q[h_{ijk}]/\mathcal{I}^D_n$ to $C_{n-1}$ sends $h_{ijn}$ to $x_{ij}$, so we need to show that the third family of generators
of $\mathcal{I}^D_n$ corresponds to the second and third families of generators of $\mathcal{I}^C_n$.
Indeed, $h_{ijn}^2$ is sent to $x_{ij}^2$, and when $i,j,k\in[n-1]$, $h_{ijk}^2$ is sent to $2(x_{ij}x_{jk} + x_{jk}x_{ki} + x_{ki}x_{ij})$
plus elements of the ideal generated by the first two families of generators of $\mathcal{I}^C_n$.

Finally, consider the $\mathfrak{S}_n$-equivariant maps
$$\varphi:\Q[h_{ijk}]\to \Q[x_{ij}]\and
\psi:\Q[x_{ij}]\to \Q[h_{ijk}]$$
given by $$\varphi(h_{ijk})= x_{ij} + x_{jk} + x_{ki}\and \psi(x_{ij}) = \frac 1 n \sum_{k\notin\{i,j\}}h_{ijk}.$$
It is a simple calculation to check that $\varphi(\mathcal{I}^D_n) \subset \mathcal{J}^D_n$
and $\psi(\mathcal{J}^D_n) \subset \mathcal{I}^D_n$, so these maps descend to maps
$$\bar\varphi:\Q[h_{ijk}]/\mathcal{I}^D_n\to \Q[x_{ij}]/\mathcal{J}^D_n\and
\bar\psi:\Q[x_{ij}]/\mathcal{J}^D_n\to \Q[h_{ijk}]/\mathcal{I}^D_n.$$
To see that they are mutually inverse, we note that
\begin{eqnarray*}
\bar\psi\circ\bar\varphi(h_{ijk}) &=& \frac 1 n \sum_{p\notin\{i,j\}}h_{ijp} + \frac 1 n \sum_{q\notin\{j,k\}}h_{jkq} + \frac 1 n \sum_{r\notin\{k,i\}}h_{kir}\\
&=& \frac 1 n (h_{ijk} + h_{jki} + h_{kij}) + \frac 1 n \sum_{l\notin\{i,j,k\}} (h_{ijl} + h_{jkl} + h_{kil})\\
&=& \frac 3 n h_{ijk} + \frac 1 n \sum_{l\notin\{i,j,k\}} h_{ijk}\\
&=& h_{ijk},\end{eqnarray*}
and 
\begin{eqnarray*}
\bar\varphi\circ\bar\psi(x_{ij}) &=& \frac 1 n \sum_{k\notin\{i,j\}} (x_{ij} + x_{jk} + x_{ki})\\
&=& \frac 1 n \sum_{k\notin\{i,j\}} x_{ij} + \frac 1 n \sum_{k\notin\{i,j\}} x_{jk} + \frac 1 n \sum_{k\notin\{i,j\}} x_{ki}\\
&=& \frac{n-2}{n} x_{ij} - \frac 1 n x_{ji} - \frac 1 n x_{ji}\\
&=& x_{ij}.
\end{eqnarray*}
This completes the proof.
\end{proof}

The following presentation of $OT_n$ appears in \cite[Section 2.1]{MPY}, where it is proved using
\cite[Theorem 4]{PS} and \cite[Proposition 2.7]{ST-OT}.

\begin{theorem}\label{OT-pres}
There exists an isomorphism $OT_n \cong \Q[x_{ij}]/\mathcal{I}^{OT}_n$, where $\Q[x_{ij}]$
is the polynomial ring with generators $x_{ij}$ for all distinct $i, j\in[n]$ and $\mathcal{I}^{OT}_n$ is the ideal generated by the following
families of relations:
\begin{itemize}
\item $x_{ij} + x_{ji}$ for all $i, j$ distinct
\item $x_{ij}x_{jk} + x_{jk}x_{ki} + x_{ki}x_{ij}$ for all $i,j,k$ distinct.
\end{itemize}
The group $\mathfrak{S}_n$ acts by permuting the indices.
\end{theorem}

As in Example \ref{R}, let $R_n = \Q[z_1,\ldots,z_n]/\langle z_1+\cdots +z_n\rangle$, with its natural grading and action of $\mathfrak{S}_n$.
We have an $\mathfrak{S}_n$-equivariant homomorphism $\varphi_n:R_n\to OT_n$ given by $\varphi(z_i) = \sum_{j\neq i} x_{ij}$,
which makes $OT_n$ into a graded $R_n$-module.  The following theorem says that the ring $M_n$ is isomorphic to the quotient of $OT_n$
by the ideal generated by the elements $\varphi(z_i)$ \cite[Section 2.1]{MPY}.

\begin{theorem}\label{M-pres}
There exists an $\mathfrak{S}_n$-equivariant isomorphism $M_n \cong \Q[x_{ij}]/\mathcal{I}^M_n$, where $\Q[x_{ij}]$
is the polynomial ring with generators $x_{ij}$ for all distinct $i, j\in[n]$ and 
$$\mathcal{I}^M_n = \mathcal{I}^{OT}_n + \big\langle \varphi(z_i)\mid i\in [n]\big\rangle.$$
\end{theorem}

\begin{remark}
Theorem \ref{M-pres} is proved by using \cite[Corollary 4.5]{TP08} to identify $OT_n$ with the torus equivariant intersection
cohomology of the hypertoric variety $X_n$, and $R_n$ with the torus equivariant cohomology of a point.
Because everything is concentrated in even degrees, the ordinary intersection cohomology is obtained
from the equivariant intersection cohomology by killing the action of the positive degree classes in the equivariant cohomology
of a point.
\end{remark}

\begin{remark}
Looking at Theorems \ref{C-pres} and \ref{OT-pres}, we see that
$C_n$ is isomorphic to the Artinian Orlik--Terao algebra of the Coxeter arrangement associated with $\mathfrak{sl}_n$,
as predicted by Remark \ref{AOTC}.
Geometrically, this reflects the fact that the locus of $X_n$ on which the torus acts freely has quotient space homeomorphic to $\Conf(n,\R^3)$,
and the map from $OT_n$ to $C_n$ may be identified with the restriction map in torus equivariant intersection cohomology.
\end{remark}

\begin{remark}
Consider the quotient of the polynomial ring $\Q[x_{ij},t]$ by the ideal generated by the following families of relations:
\begin{itemize}
\item $x_{ij} + x_{ji}$ for all $i, j$ distinct
\item $\sum_{j\neq i} x_{ij}$ for all $i$
\item $(x_{ij} + x_{jk} + x_{ki})^2 - t(x_{ij}^2 + x_{jk}^2 + x_{ki}^2)$ for all $i,j,k$ distinct.
\end{itemize}
If we specialize at $t=0$, we obtain the algebra $D_n$ by Theorem \ref{two}.  If we specialize at $t=1$, 
we obtain the algebra $M_n$ by Theorem \ref{M-pres}.  It is tempting to guess that this ring is free as a module over $\Q[t]$,
which would imply Conjecture \ref{mpy}.  However, computer calculations reveal that this is not the case.  The dimension of
a generic specialization is smaller than those of the specializations at $t=0$ or $t=1$, both of which are equal to $(n-1)!$
\cite[Remarks 2.1 and 2.4]{MPY}.
\end{remark}

\bibliography{./symplectic}
\bibliographystyle{amsalpha}

\end{document}